\documentclass[11pt]{article}

\usepackage[margin=1in]{geometry}

\usepackage{booktabs}
\usepackage{enumitem}

\usepackage{hyperref}
\hypersetup{
    colorlinks=true,     % false: boxed links; true: colored links
    linkcolor=blue,      % color of internal links
    citecolor=blue,      % color of links to bibliography
    filecolor=blue,      % color of file links
    urlcolor=blue        % color of external links
}

\usepackage{amsthm}
\usepackage{amsmath}
\usepackage{amssymb}

\usepackage{multirow}

\newtheorem{proposition}{Proposition}
\newtheorem{theorem}{Theorem}
\newtheorem{lemma}{Lemma}

\newtheorem{example}{Example}
\newtheorem{remark}{Remark}

\newtheorem{problem}{Problem}

\renewcommand{\S}{\mathbf{H}}

\newcommand{\psd}{\succeq}
\newcommand{\nsd}{\preceq}
\newcommand{\pd}{\succ}
\newcommand{\RR}{\mathbb{R}}
\newcommand{\CC}{\mathbb{C}}

\newcommand{\tr}{\textup{tr}}

\newcommand{\hypo}{\textup{hyp}}
\newcommand{\epi}{\textup{epi}}

\DeclareMathOperator{\vecm}{vec}

\title{Lieb's concavity theorem, matrix geometric means,\\ and semidefinite optimization}
\author{Hamza Fawzi\thanks{
Department of Applied Mathematics and Theoretical Physics, University of Cambridge, UK. Email: \texttt{h.fawzi@damtp.cam.ac.uk}
    } \and James Saunderson\thanks{Department of Electrical and Computer Systems Engineering, 
	Monash University, VIC 3800, Australia. Email: \texttt{james.saunderson@monash.edu}}}

\date{April 13, 2020}

\begin{document}
\maketitle

\begin{abstract}
A famous result of Lieb establishes that the map $(A,B) \mapsto \tr\left[K^* A^{1-t} K B^t\right]$ is jointly concave in the pair $(A,B)$ of positive definite matrices, where $K$ is a fixed matrix and $t \in [0,1]$. In this paper we show that Lieb's function admits an explicit semidefinite programming formulation for any rational $t \in [0,1]$.
%This means that convex optimization problems involving Lieb's function can be solved efficiently using off-the-shelf solvers for semidefinite programming.
%Our result can be used to solve (approximately) optimization problems involving von Neumann entropy or the quantum relative entropy using off-the-shelf solvers for semidefinite programming.
Our construction makes use of a semidefinite formulation of weighted matrix geometric means.  We provide an implementation of our constructions in Matlab.
\end{abstract}


%\tableofcontents

\section{Introduction}
\label{sec:introduction}
% !TEX root = sdp_rational_lieb_arXiv_v3.tex

In 1973 Lieb \cite{lieb1973convex} proved the following fundamental theorem.
\begin{theorem}[Lieb]
\label{thm:lieb}
Let $K$ be a fixed matrix in $\CC^{n\times m}$. Then for any $t \in [0,1]$, the map
\begin{equation}
\label{eq:liebmap}
(A,B) \mapsto \tr\left[K^* A^{1-t} K B^t\right]
\end{equation}
is jointly concave in $(A,B)$ where $A$ and $B$ are respectively $n\times n$ and $m\times m$ Hermitian positive definite matrices.
% Then \eqref{eq:liebmap} is jointly concave (resp. convex) in $(A,B)$ for $t \in [0,1]$ (resp. for $t \in [-1,0] \cup [1,2]$).
\end{theorem}
This theorem plays a fundamental role in quantum information theory and was used for example to establish convexity of the quantum relative entropy as well as strong subadditivity \cite{liebruskaiproofssa}.
%In this paper we are interested in the question of solving convex optimization problems involving Lieb's function. We show that such problems can be formulated as \emph{semidefinite programs}, and thus can be solved using any off-the-shelf solver for semidefinite programming.
% In order to do so we give an explicit representation of Lieb's function using semidefinite programming.
%In this paper we show that optimization problems involving Lieb's function can be effectively solved using \emph{semidefinite programming}, when $t$ is a rational number.
%The mere concavity result of Lieb cannot be directly used 
%\red{Papers of Venkat and Pari \cite{chandrasekaran2013conic} and Faybusovich \cite{faybusovich2014matrix}.}
In this paper we give an explicit representation of Lieb's function using semidefinite programming when $t$ is a rational number. More precisely we prove:
\begin{theorem}
\label{thm:main-lieb-intro}
Let $K$ be a fixed matrix in $\CC^{n\times m}$ and let $t = p/q$ be any rational number in $[0,1]$. Then the convex set
\[
\left\{ (A,B,\tau) : \tr\left[K^* A^{1-t} K B^t\right] \geq \tau \right\}
\]
has a semidefinite programming representation with at most $2 \lfloor \log_2 q\rfloor + 3$ linear matrix inequalities of size at most $2nm\times 2nm$.
\end{theorem}
Semidefinite programming is a class of convex optimization problems that can be solved in polynomial-time and that is supported by many existing numerical software packages. Having a semidefinite programming formulation of a function allows us to combine it with a wide family of other semidefinite representable functions and constraints, and solve the resulting problem to global optimality. In fact we have implemented our constructions in the Matlab-based modeling language CVX \cite{cvx} and we are making them available online on the webpage \begin{center}\url{http://www.damtp.cam.ac.uk/user/hf323/lieb_cvx.html}.\end{center}
% and combine it with other semidefinite representable functions

\paragraph{Matrix geometric means} Our proof of Theorem \ref{thm:main-lieb-intro} relies crucially on the notion of \emph{matrix geometric mean}. Given $t \in [0,1]$ and positive definite matrices $A$ and $B$, the $t$-\emph{weighted matrix geometric mean} of $A$ and $B$ denoted interchangeably by $G_t(A,B)$ or $A \#_t B$ is defined as:
\begin{equation}
\label{eq:defmatrixgeomean}
G_t(A,B) = A \#_t B := A^{1/2}\left(A^{-1/2}BA^{-1/2}\right)^t A^{1/2}.
\end{equation}
Note that when $A$ and $B$ are scalars (or commuting matrices) this formula reduces to the simpler expression $A^{1-t} B^t$. Equation \eqref{eq:defmatrixgeomean} constitutes a generalization of the geometric mean to noncommuting matrices and satisfies many of the properties that are expected from a mean operation \cite{kubo1980means,bhatiaPsdMatrices}. One remarkable property of the matrix geometric mean is that it is \emph{matrix concave}: if $t \in [0,1]$, then for any pair $X = (A_1,B_1)$ and $Y = (A_2,B_2)$ we have:
\[
G_t\left( \frac{X + Y}{2} \right) \succeq \frac{1}{2} (G_t(X) + G_t(Y))
\]
where $\succeq$ indicates the L\"{o}wner partial order on Hermitian matrices (i.e., $A \succeq B \Leftrightarrow A-B$ positive semidefinite).
This fact can be used to give a simple proof of Lieb's concavity theorem, see e.g., \cite{simplestprooflieb}. The matrix geometric mean was recently shown in \cite{sagnol2013semidefinite} to have a semidefinite programming formulation. More precisely Sagnol showed that for any rational $t=p/q \in [0,1]$ the convex set
\begin{equation}
\hypo_t := 
\left\{ (A,B,T) \in \S^n_{++} \times \S^n_{++} \times \S^n : G_t(A,B) \succeq T \right\}
\end{equation}
has a semidefinite programming representation with at most $O(\log_{2}(q))$ linear matrix inequalities of size $2n\times 2n$.
In this paper we show how an SDP representation of the matrix geometric mean can be used to get an SDP representation of Lieb's function as well as numerous other convex/concave functions. Table \ref{tbl:functions} summarizes the functions we consider in this paper, together with the size of the representations.

We only became aware of the result by Sagnol \cite{sagnol2013semidefinite}
after the first preprint of this paper appeared. As such, our alternative
approach to constructing an SDP description of the matrix geometric mean, and a
proof of its correctness, is included in Appendix~\ref{app:construction} of
this paper.  Our construction has the same size (Theorem
\ref{thm:mgeomean-sdp}) as Sagnol's, and extends to the regime $t \in [-1,0]
\cup [1,2]$ for which $G_t$ is matrix convex. Furthermore, our code, which is
available online, is based on the construction in
Appendix~\ref{app:construction}.

\begin{table}[ht]
\centering
{\small
\begin{tabular}{p{5cm}p{5.5cm}p{5cm}}
\toprule
Function & Properties & Size of SDP description ($t=p/q$)\\ 
\midrule
Matrix geometric mean\newline
$(A,B) \mapsto A \#_t B$
&
matrix concave for $t \in [0,1]$ \newline
matrix convex for $t \in [-1,0]\cup [1,2]$
&
$O(\log_2 q)$ LMIs of size $2n$\newline
(Theorem \ref{thm:mgeomean-sdp})\newline
See also \cite{sagnol2013semidefinite}.
%$(2\log_2 q+2, 2n) \oplus (1, n)$.
\\[0.8cm]

Lieb-Ando function\newline
$(A,B) \mapsto \tr\left[K^* A^{1-t} K B^t\right]$\newline
\footnotesize{($K \in \CC^{n\times m}$ fixed)}
&
concave for $t \in [0,1]$ \newline
convex for $t \in [-1,0]\cup [1,2]$
&
%$(2\log_2 q+2, 2nm) \oplus (1, nm)$ \newline
$O(\log_2 q)$ LMIs of size $2nm$\newline
(Theorem \ref{thm:main-lieb})\\[1.15cm]

$A\mapsto \tr\left[(K^* A^t K)^{1/t} \right]$\newline
\footnotesize{($K \in \CC^{n\times m}$ fixed)}
&
concave for $t \in [-1,1]\setminus \{0\}$\newline
convex for $t \in [1,2]$
&
$O(\log_2 q)$ LMIs of size $2nm$ \newline
(Theorem \ref{thm:CarlenLieb-LMI})\\[1.15cm]

Tsallis entropy\newline
$A \mapsto \frac{1}{t} \tr\left[A^{1-t} - A\right]$
%Umegaki relative entropy\newline
%$(A,B) \mapsto \tr\left[A (\log A - \log B)\right]$
&
concave for $t \in [0,1]$\newline
converges to von Neumann entropy $S(A)$ when $t \rightarrow 0$
&
$O(\log_2 q)$ LMIs of size $2n$\newline
(Remark \ref{rem:tsallisentr})
\\[1.15cm]

Tsallis relative entropy\newline
$(A,B) \mapsto \frac{1}{t} \tr\left[A - A^{1-t} B^t\right]$
%Umegaki relative entropy\newline
%$(A,B) \mapsto \tr\left[A (\log A - \log B)\right]$
&
convex for $t \in [0,1]$\newline
converges to relative entropy $S(A\|B)$ when $t \rightarrow 0$
&
$O(\log_2 q)$ LMIs of size $2n^2$\newline
(Remark \ref{rem:tsallisentr})
\\

%Tsallis relative operator\newline entropy \cite{furuichi2005note}\newline
%$(A,B) \mapsto \frac{1}{t} (A - A \#_t B)$
%Relative operator entropy\newline
%$(A,B) \mapsto A^{1/2} \log\left(A^{-1/2} B A^{-1/2}\right) A^{1/2}$
%& matrix convex for $t \in [0,1]$ & \\

%BS relative entropy & $(A,B) \mapsto \tr\left[A \log(B^{-1/2} A B^{-1/2})\right]$ & convex & ... \\
\bottomrule
%\midrule[.5pt]
%\multicolumn{3}{l}{\textsuperscript{*}\footnotesize{The notation $(k_1,n_1)\oplus (k_2,n_2)$ indicates $k_1$ LMIs of size $n_1$ and $k_2$ LMIs of size $n_2$.}}
\end{tabular}  
}
\label{tbl:functions}
\caption{List of functions with SDP formulations considered in this paper.}
\end{table}
%\paragraph{Free semidefinite representations}

\paragraph{Implications for quantum relative entropy and related functions}
Our results can be used to solve, approximately, \emph{quantum relative entropy programs} \cite{chandrasekaran2013conic} using semidefinite programming. The quantum relative entropy function is defined as:
\[
S(A\|B) = \tr\left[A(\log A - \log B)\right]
\]
where $A$ and $B$ are positive definite matrices. 
%This function generalizes the classical notion of relative entropy to density operators.
 It is a simple corollary of Lieb's theorem that $S$ is jointly convex in $(A,B)$. Indeed this follows from observing that:
\begin{equation}
\label{eq:relentlim}
S(A\|B) = \lim_{t\rightarrow 0^+} \frac{1}{t} \tr\left[A - A^{1-t} B^t\right]
\end{equation}
where we used the fact that for any matrix $X \succ 0$:
\[
\log X = \lim_{t\rightarrow 0} \frac{1}{t} (X^t - I).
\]
Identity \eqref{eq:relentlim} together with the semidefinite programming representation of Lieb's function can be used to get SDP approximations of the relative entropy function $S(A\|B)$ to arbitrary accuracy, by choosing $t$ small enough. 
%In fact the quantity on the right-hand side of \eqref{eq:relentlim} is known as the Tsallis quantum relative entropy \cite{abe2003nonadditive} (see also \cite{furuichi2004fundamental}).
Unfortunately however, the convergence of $S_t(A\|B) := \frac{1}{t}\tr\left[A - A^{1-t} B^t\right]$ to $S(A\|B)$ is slow (it is in $O(t)$) and obtaining decent approximations of $S(A\|B)$ thus requires to use very small values of $t$. While the size of the SDP descriptions of $S_t(A\|B)$ grows only like $\log(1/t)$, we observed that standard numerical algorithms to solve these SDPs become numerically ill-conditioned as $t$ gets close to 0. There exist however other methods to obtain approximations of $S(A\|B)$ that converge much faster and are better behaved numerically and these methods are discussed in \cite{logapprox}.

\paragraph{Related works} It is well-known that the scalar functions $(x,y) \mapsto x^{1-t} y^t$ admit second-order cone representations when $t$ is a rational number \cite[Chapter 3]{ben2001lectures}. The SDP representation of the matrix geometric mean can be seen as a matrix generalization of such results.
The authors of \cite{helton2015free} give a free semidefinite representations of the matrix power functions $X\mapsto X^t$ for rational $t \in [-1,2]$, however it seems that they were not aware of the paper by Sagnol \cite{sagnol2013semidefinite} since such a representation already appears in this work. Furthermore the construction in \cite{sagnol2013semidefinite} is in some cases smaller than \cite{helton2015free}: for general rational $t=p/q \in [0,1]$ the construction in \cite{sagnol2013semidefinite} has size $O(\log_2(q))$ whereas in some cases the construction in \cite{helton2015free} requires a number of LMIs that grows linearly with $q$.
The authors of \cite{helton2015free} also mentioned that certain multivariate versions of the matrix power function fail to have semidefinite representations. Working in the setting of geometric means, and then tensor products, seems to give one natural extension to the multivariate case (see Remark \ref{rem:extension_lieb}).

\paragraph{Outline} In Section~\ref{sec:preliminaries} we set up the basic notations and terminology for the paper and in Section~\ref{sec:sdp_representations} we prove the main results of the paper giving SDP representations of the functions given in Table~\ref{tbl:functions}.

\section{Preliminaries}
\label{sec:preliminaries}
% !TEX root = sdp_rational_lieb_LAA.tex

In this section we introduce basic notation and terminology used throughout the paper.
Let $\S^n$ be the space of $n\times n$ Hermitian matrices, $\S_+^n\subset \S^n$ the cone of $n\times n$ Hermitian positive semidefinite matrices and $\S_{++}^n$ the 
cone of $n\times n$ strictly positive definite matrices.
We use the notation $X \psd Y$ if $X-Y$ is positive semidefinite, and $X \pd Y$ if $X-Y$ is positive definite.
Suppose $C$ is a convex set and $f:C\rightarrow \S^n$. We say that $f$ is \emph{$\S_+^n$-convex} if the $\S_+^n$-epigraph
\[ \epi_{\S_+^n}(f) := \{(X,T)\in C\times \S^n: f(X) \nsd T\}\]
is a convex set. Similarly $f$ is \emph{$\S_+^n$-concave} if the $\S_+^n$-hypograph
\[ \hypo_{\S_+^n}(f) := \{(X,T) \in C\times \S^n: f(X) \psd T\}\]
is a convex set.

\paragraph{Semidefinite representations} A semidefinite program is an optimization problem that takes the form
\[
\begin{array}{ll}
\text{maximize} & \langle b, y \rangle\\
\text{subject to} & A_0 + y_1 A_1 + \dots + y_n A_n \succeq 0
\end{array}
\]
where $y \in \RR^n$ is the optimization variable, $b$ is a fixed vector in $\RR^n$ and $A_0,A_1,\dots,A_n \in \S^m$ are fixed $m\times m$ Hermitian matrices. The condition
\[ A_0 + y_1 A_1 + \dots + y_n A_n \succeq 0 \]
is known as a \emph{linear matrix inequality} (LMI) of size $m$. We will say that a convex set $C$ has a \emph{SDP representation} if it can be expressed using LMIs (we allow for lifting variables). To evaluate the size of a semidefinite representation we record the number of LMIs of each size. For example consider the following convex set $H$:
\[ H = \{ (x_1,x_2,x_3) \in \RR^3 \; : \; x_1,x_2,x_3 \geq 0\text{ and } x_1 x_2 x_3 \geq 1\}. \]
One can show that $H$ admits the following SDP representation:
\begin{equation}
\label{eq:Hlift}
\begin{aligned}
H &= \left\{ (x_1,x_2,x_3) \in \RR^3 \; : \; \exists y,z  \quad \text{ s.t. } \quad
\begin{bmatrix} x_1 & y\\ y & x_2 \end{bmatrix} \succeq 0, \;
\begin{bmatrix} x_3 & z\\ z & 1 \end{bmatrix} \succeq 0, \;
\begin{bmatrix} y & 1\\ 1 & z \end{bmatrix} \succeq 0
\right\}.
\end{aligned}
\end{equation}
This SDP representation consists of 3 LMIs of size 2 each.

\paragraph{Kronecker products and their properties} If $A \in \CC^{m\times n}$ we denote by $A^* \in \CC^{n\times m}$ the conjugate transpose of $A$.
The \emph{Kronecker product} of $A\in \CC^{m_1\times n_1}$ and $B\in \CC^{m_2\times n_2}$ is the $\CC^{m_1m_2\times n_1n_2}$ matrix $A \otimes B$ with 
\[ [A\otimes B]_{(i,k)(j,\ell)} = A_{ij}B_{k\ell}\qquad\text{for $1\leq i\leq n_1, 1\leq j \leq n_2, 1\leq k \leq m_1, 1\leq \ell\leq m_2$}.\]
If $A,B,C,D$ are matrices of compatible dimensions then $(A \otimes B)(C \otimes D) = (AC\otimes BD)$ and $(A \otimes B)^* = A^* \otimes B^*$. 
Suppose $A\in \S^n$ and $B\in \S^m$ are Hermitian matrices with eigenvalue decompositions $A = U\Lambda_A U^*$ and $B = V \Lambda_B V^*$ where $U,V$ are unitary
matrices and $\Lambda_A$ and $\Lambda_B$ are diagonal. Then $U\otimes V$ is unitary and $\Lambda_A \otimes \Lambda_B$ is diagonal and so
\[ A \otimes B = (U\otimes V) (\Lambda_A\otimes \Lambda_B) (U \otimes V)^*\]
is an eigenvalue decomposition of $A\otimes B$.

\section{SDP representations}
\label{sec:sdp_representations}
% !TEX root = sdp_rational_lieb_LAA.tex

This is the main section of the paper where we describe the SDP representations of the various functions in Table \ref{tbl:functions}.

\subsection{Matrix geometric mean}

We first consider the SDP representation of the matrix geometric mean. Recall that the \emph{$t$-weighted geometric mean} $G_{t}: \S_{++}^n\times \S_{++}^n\rightarrow \S_{++}^n$ is defined by
\[ G_{t}(A,B) = A \#_t B := A^{1/2}\left(A^{-1/2}BA^{-1/2}\right)^tA^{1/2}.\]
It is known \cite{bhatiaPsdMatrices} that $G_t$ is matrix concave for $t \in [0,1]$ and is matrix convex for $t \in [-1,0] \cup [1,2]$. We denote by $\hypo_t$ and $\epi_t$ the matrix hypograph and matrix epigraph of $G_t$ respectively:
\[
\hypo_{t} = \left\{ (A,B,T) \in \S^n_{++} \times \S^n_{++} \times \S^n : A \#_t B \succeq T \right\}
\]
for $t \in [0,1]$, and
\[
\epi_t = \left\{ (A,B,T) \in \S^n_{++} \times \S^n_{++} \times \S^n : A \#_t B \preceq T \right\}
\]
for $t \in [-1,0]\cup[1,2]$. These notations do not keep track of the dimension $n$ explicitly but this omission should not cause any confusion.

The next theorem shows that the matrix geometric mean $G_t$ for rational $t=p/q$ admits an SDP formulation involving $O(\log_2 q)$ LMIs of size at most $2n \times 2n$.
The case $t \in [0,1]$ was already obtained by Sagnol \cite{sagnol2013semidefinite}.
In Appendix~\ref{app:construction}, we explicitly describe our construction (on which our CVX code is based) and establish its correctness.
% to tha we implemented in the code we made available online.

%Our main theorem can be stated as follows (note that this is a restatement of Theorem \ref{thm:intro-mgeomean-sdp} from the introduction with the additional case $t \in [-1,0] \cup [1,2]$).

%Sagnol \cite{sagnol2013semidefinite} showed that for any rational $t = p/q \in [0,1]$, the convex set $\hypo_{t}$ admits a semidefinite programming formulation of with at most $O(\log_2 q)$ LMIs of size $2n \times 2n$.

\begin{theorem}
\label{thm:mgeomean-sdp}
Let $p,q$ be relatively prime integers with $p/q \in [-1,2]$.
\begin{itemize}
\item If $p/q \in [0,1]$ then $\hypo_{p/q}$ has a SDP description with at most $2\lfloor\log_{2}(q)\rfloor+1$ LMIs of size $2n\times 2n$ and one LMI of size $n\times n$.
\item If $p/q \in [-1,0]\cup [1,2]$ then $\epi_{p/q}$ has a SDP description with at most $2\lfloor\log_{2}(q)\rfloor+2$ LMIs of size $2n\times 2n$ and one LMI of size $n\times n$.
\end{itemize}
\end{theorem}
\begin{proof}
The construction is detailed in Appendix \ref{app:construction}.
\end{proof}
We briefly discuss qualitative differences between our construction and that of Sagnol~\cite{sagnol2013semidefinite}. 
Our construction is recursive in nature, repeatedly expressing $\hypo_{p/q}$ in terms of $\hypo_{p'/q'}$ for a `simpler' rational $p'/q'$, until reaching 
the base case of $\hypo_{1/2}$. This recursive structure makes our construction particularly natural to implement in code. Sagnol's construction 
is similar to what we would obtain if we explicitly unrolled our recursion, even though it is expressed quite differently. 
Indeed Sagnol's approach assigns variables to the vertices of a binary tree
and the LMIs describe relationships between each vertex and its children. To prove correctness of his SDP description, Sagnol requires a contraction argument 
with respect to the Riemannian metric on positive definite matrices. In contrast, the argument we give in Appendix~\ref{app:construction}
depends only on basic properties of the geometric mean and Schur complements.  

%The construction given in \cite{sagnol2013semidefinite} is only for $t=p/q \in [0,1]$. The other cases can be obtained using properties \eqref{eq:negt} and \eqref{eq:epirev} of the matrix geometric mean. 

%We also given an alternative construction (of the same size as that of \cite{sagnol2013semidefinite}) which is given in Appendix \ref{app:construction}.

\subsection{SDP description for functions in Table \ref{tbl:functions}}

In this section we show how the SDP description of the matrix geometric mean can be used to obtain an SDP description of the functions given in Table \ref{tbl:functions}.

\subsubsection{Lieb's function}

We first consider Lieb's function. The following is a restatement of Theorem \ref{thm:main-lieb-intro} from the introduction with the additional case $t \in [-1,0]\cup [1,2]$.

\begin{theorem}
\label{thm:main-lieb}
Let $K$ be a fixed matrix in $\CC^{n\times m}$ and let $t = p/q$ be any rational number in $[-1,2]$. Let $F_t(A,B) = \tr\left[K^* A^{1-t} K B^t\right]$.
\begin{itemize}
\item If $t=p/q \in [0,1]$, then $F_t$ is concave and its hypograph admits a semidefinite programming representation using at most $2 \lfloor \log_2 q\rfloor + 1$ LMIs of size $2nm\times 2nm$, one LMI of size $nm\times nm$ and one scalar inequality.
\item If $t=p/q \in [-1,0]\cup [1,2]$, then $F_t$ is convex and its epigraph admits a semidefinite programming representation using at most $2 \lfloor \log_2 q\rfloor + 2$ LMIs of size $2nm\times 2nm$, one LMI of size $nm\times nm$ and one scalar inequality.
\end{itemize}
\end{theorem}
\begin{proof}
To prove this theorem we use the well-known relationship between $F_t$ and the matrix-valued function $L_t(A,B) = A^{1-t} \otimes \bar{B}^t$ due to Ando. In fact it is not difficult to verify that we have the following identity:
\begin{equation}
\label{eq:andotensorid}
\tr\left[K^* A^{1-t} K B^t\right] = \vecm(K)^* (A^{1-t} \otimes \bar{B}^t) \vecm(K)
\end{equation}
where $\vecm(K)$ is a column vector of size $nm$ obtaining by concatenating the rows of $K$ and $\bar{B}$ is the entrywise complex conjugate of $B$ (see e.g., \cite[Lemma 5.12]{carlen-notes}). Thus, if $t \in [0,1]$ we have for any real number $\tau$
\begin{equation}
\label{eq:hypolieb}
\tr\left[K^* A^{1-t} K B^t\right] \geq \tau \iff \exists T \in \S^{nm}_{++} \text{ s.t. } \begin{cases}
A^{1-t} \otimes \bar{B}^t \succeq T\\
\vecm(K)^* T \vecm(K) \geq \tau.
\end{cases}
\end{equation}
We now show how to convert \eqref{eq:hypolieb} into an SDP formulation. The key idea (see e.g., \cite{simplestprooflieb}) is to note that
\begin{equation}
\label{eq:krongeomeanid}
A^{1-t} \otimes \bar{B}^t = (A \otimes I) \#_{t} (I \otimes \bar{B})
\end{equation}
where $I$ denotes the identity matrix of appropriate size. To see why \eqref{eq:krongeomeanid} holds, note that $A\otimes I$ and $I\otimes \bar{B}$ commute and so
\[
(A\otimes I) \#_{t} (I \otimes \bar{B}) = (A\otimes I)^{1-t} (I\otimes \bar{B})^{t} \overset{(a)}{=} (A^{1-t} \otimes I) (I \otimes \bar{B}^t) \overset{(b)}{=} A^{1-t} \otimes \bar{B}^t
\]
where $(a)$ can be shown using the eigenvalue decompositions of $A\otimes I$ and $I\otimes \bar{B}$, and $(b)$ follows from the properties of the Kronecker product.
Using the SDP formulation of the matrix geometric mean (Theorem \ref{thm:mgeomean-sdp}) we can thus formulate the constraint $A^{1-t} \otimes \bar{B}^t \succeq T$ using $2\lfloor \log_2(q) \rfloor + 1$ LMIs of size $2nm\times 2nm$ and one LMI of size $nm\times nm$ (where $t=p/q$). Plugging this in \eqref{eq:hypolieb} gives us an SDP formulation of the hypograph of Lieb's function with the required size.
The case $t \in [-1,0] \cup [1,2]$ is treated in the same way.
\end{proof}

\begin{remark}
\label{rem:extension_lieb}
\begin{itemize}
\item It is straightforward to extend Theorem \ref{thm:main-lieb} to get an SDP formulation of the functions $(A,B) \mapsto A^s \otimes B^t$ where $s$ and $t$ are nonnegative numbers such that $s+t \leq 1$. It suffices to observe that
\[
A^s \otimes B^t \succeq T \quad \Longleftrightarrow \quad \exists S \in \S^{nm}_{+} \text{ s.t. } 
\begin{cases}
A^{\frac{s}{s+t}} \otimes B^{\frac{t}{s+t}} \succeq S\\
S^{s+t} \succeq T.
\end{cases}
\]
\item Similarly one can also extend Theorem \ref{thm:main-lieb} to obtain an SDP formulation of a $k$-variate generalization of the Lieb function, namely
\[
(A_1,\dots,A_k) \mapsto A_1^{t_1} \otimes \dots \otimes A_k^{t_k}
\]
where $t_1,\dots,t_k \geq 0$ are such that $t_1+\dots+t_k = 1$. To do so we simply eliminate one matrix at a time. For example in the case $k=3$ we use:
\[
A_1^{t_1}\otimes A_2^{t_2} \otimes A_3^{t_3} \succeq T
\quad \Longleftrightarrow \quad
\exists S \in \S^{n_1n_2}_+ \text{ s.t. }
\begin{cases}
A_1^{\frac{t_1}{t_1+t_2}} \otimes A_2^{\frac{t_2}{t_1+t_2}} \succeq S\\
S^{t_1+t_2} \otimes A_3^{t_3} \succeq T.
\end{cases}
\]
\end{itemize}
\end{remark}

\begin{remark}[Tsallis entropies]
\begin{itemize}
\item For $t \in [0,1]$ the \emph{Tsallis entropy} \cite{tsallis1988possible} is defined  as
\[
S_t(A) := \frac{1}{t} \tr\left[A^{1-t} - A\right].
\]
It is easy to see that $S_t(A)$ converges (from above) to the von Neumann entropy $S(A) = -\tr[A\log A]$ when $t\rightarrow0$, i.e., $S_t(A) \geq S(A)$ for any $t \in [0,1]$ and $\lim_{t\rightarrow 0} S_t(A) = S(A)$. Also note that $S_t$ is concave for all $t \in [0,1]$. 
Its hypograph,
$\left\{(A,\tau)\in \S_{++}^n\times \RR: S_t(A) \geq \tau\right\}$, can be expressed in terms of the matrix geometric mean as 
\[  
\left\{(A,\tau)\in \S_{++}^n\times \RR: \exists T\in \S^n\;\;\textup{s.t.}\;\; A \#_t I \psd T,\;\; \frac{1}{t} \tr\left[ T-A\right] \geq \tau\right\}.\]
By rewriting $A\#_t I \psd T$ using the SDP description of the matrix geometric mean (with $B=I$), we obtain a SDP description of $S_t$ (when $t=p/q$) 
having $O(\log_2 q)$ LMIs of size at most $2n$.
\item The \emph{Tsallis relative entropy} is defined for $t \in [0,1]$ as (see \cite{abe2003nonadditive} and also \cite{furuichi2004fundamental})
\[
S_t(A\|B) := \frac{1}{t}\tr\left[A - A^{1-t} B^t\right].
\]
As noted in \eqref{eq:relentlim} the Tsallis relative entropy $S_t(A\|B)$ converges to  
the quantum relative entropy $S(A\|B) = \tr\left[A(\log A - \log B)\right]$ when $t \rightarrow 0$. It is also known that convergence is from below, i.e., $S_t(A\|B) \leq S(A\|B)$ for any $t \in [0,1]$ (see e.g., \cite[Proposition 2.1]{furuichi2004fundamental}).
By choosing $K=I$ in Lieb's theorem we see that $S_t(A\|B)$ is jointly convex in $(A,B)$. Indeed the epigraph of $S_t(\cdot\|\cdot)$   
can be expressed as
\begin{multline*}
\left\{(A,B,\tau)\in \S_{++}^n\times \S_{++}^n\times \RR: S_t(A\|B) \leq \tau\right\} = \\
\left\{ (A,B,\tau)\in \S_{++}^n\times \S_{++}^n\times \RR: \exists s\in \RR\;\;\textup{s.t.}\;\;\tr\left[A^{1-t}B^t\right] \geq s,\;\; \frac{1}{t}\left[\tr(A) - s\right] \leq \tau\right\}.
\end{multline*} 

By rewriting $\tr[A^{1-t}B^t]\geq s$ using the SDP description from Theorem \ref{thm:main-lieb} (with $K=I$), we obtain 
a SDP description of $S_t(\cdot\|\cdot)$ (with $t=p/q$) having $O(\log_2 q)$ LMIs of size at most $2n^2$.
\end{itemize}
\label{rem:tsallisentr}
\end{remark}

\subsubsection{The map $A\mapsto \tr\left[(K^* A^t K)^{1/t}\right]$}

Let $K$ be a fixed $n\times m$ matrix and consider the function $\Upsilon_{t}:\S_{++}^n\rightarrow \RR$ defined by 
\[ \Upsilon_{t}(A) = \tr\left[(K^* A^t K)^{1/t}\right].\]
The following result is due to Carlen and Lieb~\cite{carlen2008minkowski} where they established the case 
$t \in [0,2]$ (the same arguments were used to prove the case $t \in [-1,0)$ in \cite{frank2013monotonicity}; 
the case $t \in (0,1]$ was first established by Epstein~\cite{epstein1973remarks}).
\begin{theorem}
	\label{thm:carlen-lieb}
	If $t\in [1,2]$ then $\Upsilon_{t}$ is convex on $\S_{++}^n$. If $t\in [-1,1]\setminus\{0\}$ then $\Upsilon_{t}$
	is concave on $\S_{++}^n$. 
\end{theorem}
Rather than working with $\Upsilon_t$, it is slightly more natural to focus on the function $t\Upsilon_t$. Since $t$ is a fixed parameter, 
this simply changes the signs of some expressions for the cases $t\in [-1,0)$. It follows directly from 
Theorem~\ref{thm:carlen-lieb} that $t\Upsilon_t$ is convex on $\S_{++}^n$ for $t\in [-1,0) \cup [1,2]$, and concave on $\S_{++}^n$ for $t\in (0,1]$. 

In this section we show how to give SDP formulations of $\hypo(t\Upsilon_{t})$ for $t\in (0,1]$ and $\epi(t\Upsilon_{t})$ for $t\in [-1,0)\cup [1,2]$
by using our SDP formulations of Lieb's function in different regimes of the parameters.
Our SDP formulations %(equations \eqref{eq:hypotUpsilont} and~\eqref{eq:epitUpsilont} to follow) 
rely on variational expressions for $t\Upsilon_t$ (equations \eqref{eq:carlen-max} and~\eqref{eq:carlen-min} to follow) 
established in \cite{carlen2008minkowski} (see also \cite{carlen-notes}). 
We include a proof of these variational descriptions, for completeness, en route to our expressions
for the hypograph/epigraph of $t\Upsilon_t$ in terms of Lieb's function 
(equations \eqref{eq:hypotUpsilont} and~\eqref{eq:epitUpsilont} to follow).
\begin{lemma} Let $A \in \S_{++}^n$ and $t \in [-1,2] \setminus \{0\}$.
\begin{itemize}
\item If $t \in (0,1]$ then 
	\begin{equation} 
	\label{eq:carlen-max}
		t\Upsilon_{t}(A) = \max_{X\in \S_{++}^m} \tr\left[K^* A^{t}KX^{1-t}\right] - (1-t)\tr[X].
	\end{equation}
	Hence 
	\begin{equation}
\label{eq:hypotUpsilont}
 \hypo(t\Upsilon_{t}) = \left\{(A,\tau)\in \S_{++}^n\times \RR :
\exists X \in \S_{++}^m \text{ s.t. } \tr\left[K^* A^t K X^{1-t}\right] - (1-t) \tr[X] \geq \tau \right\}.
% \exists X\in \S^m,\; Z\in \S^{mn}\;\;\text{s.t.}\;\;
%	X \psd 0,\;\ A^{t}\otimes X^{1-t} \psd Z,\right.\\
%	& \qquad\qquad\qquad\qquad\qquad\qquad\qquad\qquad\qquad \left.\;\;\vecm(K)^* Z \vecm(K) \geq \tau +  (1-t)\tr[X]\right\}.
\end{equation}
\item If $t\in [-1,0)\cup[1,2]$ then 
	\begin{equation}	
	\label{eq:carlen-min} 
		t\Upsilon_{t}(A) = \min_{X \in \S_{++}^m} \tr\left[K^* A^{t}KX^{1-t}\right] - (1-t)\tr[X].
	\end{equation}
	Hence
	\begin{equation}
\label{eq:epitUpsilont} \epi(t\Upsilon_{t}) = \left\{(A,\tau)\in \S_{++}^n\times \RR :
\exists X \in \S_{++}^m \text{ s.t. } \tr\left[K^* A^t K X^{1-t}\right] - (1-t) \tr[X] \leq \tau \right\}.
%\exists X \in \S^m,\; Z\in \S^{mn}\;\;\text{s.t.}\;\;
%	X \psd 0,\;\; A^t\otimes X^{1-t} \nsd Z,\right.\\
%	&\qquad\qquad\qquad\qquad\qquad\qquad\qquad\qquad\qquad\left.\;\; \vecm(K)^* Z\vecm(K)\leq \tau  + (1-t)\tr[X] \right\}.
	\end{equation}
\end{itemize}
\end{lemma}
\begin{proof}
	First observe that if $t\in [0,1]$ and $y,z > 0$ then the arithmetic-mean geometric-mean inequality gives
	\begin{equation}
	\label{eq:amgm1}
	ty \geq y^t z^{1-t} - (1-t)z
	\end{equation}
	for all $y,z>0$. % with equality if and only if $z=y$. 
	%Hence
	%\[ sx = \sup_{y>0} x^sy^{1-s} - (1-s)y.\]
	If $t\in [1,2]$ and $y,z >0$ then the arithmetic-mean geometric-mean inequality gives
	$\frac{1}{t}y^t + \frac{t-1}{t}z^t \geq yz^{t-1}$. Rearranging gives
	% with equality if and only if $z=y$. Rearranging gives 
	\begin{equation}
	\label{eq:amgm2}
	 ty \leq y^tz^{1-t}-(1-t)z
	\end{equation}
	for all $y,z>0$. % with equality if and only if $z=y$. 
	%and so
	%\[ sx = \inf_{y>0} x^sy^{1-s} - (1-s)y.\]
	If $t\in [-1,0]$ and $a,b>0$ then $s = 1-t\in [1,2]$. Hence $sa + (1-s)b \leq a^s b^{1-s}$. % with equality if and only if $a=b$.
	Putting $y=b$ and $z=a$ we obtain that \eqref{eq:amgm2} also holds when $t \in [-1,0]$ for all $y,z > 0$. % with equality if and only if $z=y$.

	We can apply inequalities \eqref{eq:amgm1} and \eqref{eq:amgm2} to the eigenvalues of the positive definite commuting matrices
	$Y = (K^* A^t K)^{1/t} \otimes I$ and $Z = I\otimes \bar{X}$ (with $A\in \S_{++}^{n}$ and $X \in \S_{++}^m$). Doing so we see that if $t\in (0,1]$ then 
	\[ t(K^* A^t K)^{1/t} \otimes I \psd K^* A^t K \otimes \bar{X}^{1-t} - (1-t)(I\otimes \bar{X})\]
	for all $A \in \S_{++}^n$ and all $X\in \S_{++}^m$. 
%	with equality if and only if $X = (K^* A^t K)^{1/t}$. {\small [\red{We have equality for matrices when $Y=Z$, i.e., $(K^* A^t K)^{1/t} \otimes I = I \otimes X$ which is not the same as $X = (K^* A^t K)^{1/t}$. But actually we do not need equality case for matrices, we just need it for the scalar expression which obviously holds $X=(K^* A^t K)^{1/t}$.}]}
	Similarly if $t\in [-1,0)\cup[1,2]$ then 
	\[ t(K^* A^t K)^{1/t} \otimes I \nsd K^* A^t K \otimes \bar{X}^{1-t} - (1-t)(I\otimes \bar{X})\]
	for all $A \in \S_{++}^n$ and all $X\in \S_{++}^m$. 
	% with equality if and only if $X = (K^* A^t K)^{1/t}$. 
	If we apply the map $\S^{m^2} \ni M \mapsto \vecm(I)^* M \vecm(I)$ to both sides
	of these matrix inequalities and use identity \eqref{eq:andotensorid} we get that for $t\in [0,1)$ 
	\[ t\Upsilon_{t}(A) \geq \tr\left[K^* A^{t}K X^{1-t}\right] - (1-t)\tr[X] \]
	for all $A \in \S_{++}^n$ and all $X\in \S_{++}^m$, and for $t\in [-1,0)\cup[1,2]$
	\[ t\Upsilon_{t}(A) \leq \tr\left[K^* A^{t}K X^{1-t}\right] - (1-t)\tr[X]\]
	for all $A \in \S_{++}^n$ and all $X\in \S_{++}^m$.
To ensure that the variational formulas~\eqref{eq:carlen-max} and~\eqref{eq:carlen-min} hold, one simply checks 
that putting $X = (K^*AK)^{1/t}$ gives equality in both cases. The descriptions of $\hypo(t\Upsilon_t)$ for $t\in (0,1]$ 
and $\epi(t\Upsilon_t)$ for $t\in [-1,0)\cup[1,2]$ are direct consequences of~\eqref{eq:carlen-max} and~\eqref{eq:carlen-min}
respectively. 
\end{proof}
% These variational characterizations can be translated directly into SDP formulations of $t\Upsilon_t$:
When $t$ is rational, each of the convex sets \eqref{eq:hypotUpsilont} and \eqref{eq:epitUpsilont} can be expressed explicitly in terms of LMIs by using the SDP description of Lieb's function from Theorem~\ref{thm:main-lieb}. The following summarizes the size of these descriptions. 
\begin{theorem}
\label{thm:CarlenLieb-LMI}
Let $p,q$ be relatively prime integers such that $p/q \in [-1,2] \setminus \{0\}$.
\begin{itemize}
\item If $t=p/q\in (0,1]$ then $\hypo(t\Upsilon_t)$ has a SDP description 
with at most $2\lfloor \log_2(q) \rfloor+1$ LMIs of size $2mn \times 2mn$, one LMI of size $mn \times mn$, and one scalar inequality.
\item If $t=p/q \in [-1,0) \cup [1,2]$ then $\epi(t\Upsilon_t)$ 
has a SDP description with at most $2\lfloor \log_2(q) \rfloor+2$ LMIs of size $2mn \times 2mn$, one LMI of size $mn \times mn$, and one scalar inequality.
\end{itemize}
\end{theorem}

%\section{SDP representation of the matrix geometric mean}
%\label{sec:construction}
%\input{construction}

%\section{SDP representations for other concave/convex functions}
%\label{sec:other_functions}
%\input{other_functions}

\section{Numerical experiments}
\label{sec:numerical_experiments}
% !TEX root = sdp_rational_lieb_LAA.tex

In this section we present some numerical results for the semidefinite programming representations given in this paper.

\subsection{Maximum entropy problem}

We consider maximum entropy optimization problems of the form
\begin{equation}
\label{eq:maxentropy}
\begin{array}{ll}
\text{maximize} & S_t\left(\sum_{i=1}^s p_i M_i\right)\\
\text{subject to} & p \geq 0, \; \sum_{i=1}^s p_i = 1
\end{array}
\end{equation}
where $M_1,\dots,M_s$ are fixed $n\times n$ positive semidefinite matrices of
trace one, and $S_t(A) = \frac{1}{t}\tr[A^{1-t} - A]$ is the Tsallis entropy
considered in Remark \ref{rem:tsallisentr}. We used CVX \cite{cvx,gb08} to
formulate the problem with the new function \verb|tsallis_entr| available in
our package \cite{lieb-cvx-code}. The listing below shows the Matlab code used
to solve \eqref{eq:maxentropy} in the case $s=3$:
\begin{verbatim}
% M1,M2,M3 are three positive semidefinite matrices
cvx_begin
    variable p(3);
    maximize (tsallis_entr(p(1)*M1+p(2)*M2+p(3)*M3,1/8));
    subject to
       p >= 0;
       sum(p) == 1;
cvx_end
\end{verbatim}
Note that \verb|tsallis_entr| is automatically recognized by CVX as being a
concave function of its first argument (the matrix). The second argument to
\verb|tsallis_entr| is the parameter $t$ which we set here to be $1/8$.

In Table~\ref{tbl:numerical} we present numerical results obtained with different values of $n$ (matrix
size) while fixing $s=10$. The results were obtained with the solver
SeDuMi~\cite{sturm1999using},
which is currently packaged with CVX.

% We used the value $t=1/8$ for the experiments, and the matrices $M_1,\dots,M_s$ were generated at random, with $s=10$. Table \ref{tbl:numerical} shows the results of the numerical experiments for different values of $n$ (matrix sizes).
 
%The CVX code for this example is (we consider the case $s=3$ for simplicity). The function \verb|tsallis_entr| which implements the Tsallis entropy function where the first argument is the matrix and the second argument is the parameter $t$.

\begin{table}[ht]
\centering
\begin{tabular}{cccc}
\toprule
$n$ & optimal value & time (s)\\
\midrule
\multirow{5}{*}{10} & 2.6027 & 0.638\\
 & 2.6051 & 0.501\\
 & 2.6280 & 0.292\\
 & 2.6092 & 0.250\\
 & 2.6262 & 0.260\\ \hline
\multirow{5}{*}{20} & 3.5662 & 0.521\\
 & 3.5706 & 0.519\\
 & 3.5625 & 0.523\\
 & 3.5665 & 0.516\\
 & 3.5717 & 0.516\\
 \bottomrule
 \end{tabular}
 \qquad
\begin{tabular}{cccc}
\toprule
$n$ & optimal value & time (s)\\
\midrule 
\multirow{5}{*}{30} & 4.1619 & 2.048\\
 & 4.1689 & 3.336\\
 & 4.1767 & 2.038\\
 & 4.1638 & 2.261\\
 & 4.1784 & 2.040\\ \hline
\multirow{5}{*}{40} & 4.6191 & 8.763\\
 & 4.6144 & 9.051\\
 & 4.6172 & 9.272\\
 & 4.6125 & 10.131\\
 & 4.6131 & 10.498\\
\bottomrule
\end{tabular}
\caption{Results of numerical experiments for the maximum entropy problem \eqref{eq:maxentropy} with $t=1/8$. The matrices $M_1,\dots,M_s$ were generated at random.
	Each cell (labeled by a value of $n$) corresponds to five different random trials. 
	Note that the SDP representation of \eqref{eq:maxentropy} with $t=1/8$ consists of $3$ semidefinite constraints of size $2n$ each.}
\label{tbl:numerical}
\end{table}

\subsection{Relative entropy of entanglement}

We now consider another numerical illustration of our results to compute lower bounds on the so-called \emph{relative entropy of entanglement} in quantum information theory which is used to measure the distance of a given bipartite state $\rho$ to the set of separable states. This quantity is in general intractable to compute \cite{huang2014computing} and we consider here a popular relaxation using the \emph{positive partial transpose} (PPT) criterion~\cite{peres1996separability}. This relaxation is defined in terms of the following optimization problem, where $\rho$ is a fixed positive semidefinite matrix of trace one and $\tau$ is the optimization variable:
%Since the set of separable states is in general intractable \cite{gurvits2003classical} one is led to consider the \emph{positive partial transpose} relaxation of this set.
\begin{equation}
\label{eq:REE}
\begin{array}{ll}
\underset{\tau}{\text{minimize}} & S(\rho \| \tau)\\
\text{subject to} & \tau \succeq 0, \; \tr[\tau] = 1\\
                  & \tau \in \text{PPT}.
\end{array}
\end{equation}
The constraint $\tau \in$ PPT is a linear matrix inequality constraint. We omit its precise meaning here. 
Note that the variable $\tau$ of the optimization problem \eqref{eq:REE} enters the second argument of the 
relative entropy $S$ in the cost function. As such the cost function is \emph{not} a matrix trace function 
of the form considered e.g., in \cite[Theorem 3.1]{sagnol2013semidefinite}.

If we replace the cost function $S(\rho\|\tau)$ in \eqref{eq:REE} by the
Tsallis relative entropy $S_t(\rho\|\tau)$ (for $t$ rational) the resulting optimization problem can be expressed as an SDP using the formulation given in Remark \ref{rem:tsallisentr}. Furthermore since $S_t(\rho\|\tau)\leq S(\rho\|\tau)$ the optimal value we get is always a lower bound to \eqref{eq:REE}. The function \verb|tsallis_rel_entr| available in our package \cite{lieb-cvx-code} can be used to formulate the resulting problem using CVX on Matlab. The code is shown below. (Note that
the code uses the function \verb|Tx| from the quantinf package
\cite{cubittQuantinf} to implement the PPT constraint on $\tau$.)
%If we replace the relative entropy $S$ by the Tsallis relative entropy $S_t(\rho\| \sigma) = \frac{1}{t} \tr\left[\rho - \rho^{1-t} \sigma^t\right]$.

\begin{verbatim}
na = 2;
nb = 2;
% Generate a random positive semidefinite matrix rho of size na*nb of trace one
rho = randn(na*nb,na*nb); rho = rho*rho';
rho = rho/trace(rho);
cvx_begin
    variable tau(na*nb,na*nb) hermitian;
    minimize (tsallis_rel_entr(rho,tau,2^(-8)));
    subject to
       tau == hermitian_semidefinite(na*nb);
       trace(tau) == 1;
       % Positive partial transpose constraint
       Tx(tau,2,[na nb]) == hermitian_semidefinite(na*nb);
cvx_end
\end{verbatim}
Note that CVX automatically recognizes \verb|tsallis_rel_entr| as a convex
function of its arguments $\rho$ and $\tau$. The third argument of
\verb|tsallis_rel_entr| specifies the value of $t$ to use in the definition of
Tsallis relative entropy. Here we use $t = 2^{-8}$.
% Also note that the code above uses the functions \verb|Tx| from the \verb|quantinf| package \cite{cubittQuantinf} to perform the partial transpose operation.

We now present numerical experiments where we solve the optimization problem
for random bipartite states $\rho$. We use the solver SeDuMi and compare our
results with the tailored algorithm developed in
\cite{zinchenko2010numerical,girard2015erratum} based on a cutting-plane
approach. Table~\ref{tbl:ree_results} shows the results for different matrix
sizes $n = n_A\times n_B$ (where $n_A$ and $n_B$ are the sizes of the
subsystems). Note that the cutting-plane approach of
\cite{zinchenko2010numerical,girard2015erratum} returns an interval of length
$\epsilon$ that is guaranteed to contain the optimal value of \eqref{eq:REE}
(we chose $\epsilon = 10^{-3}$ in the experiments). We see that our method
consistently gives lower bounds that are better than the cutting-plane approach
of \cite{zinchenko2010numerical,girard2015erratum} in a fraction of the time it
takes.

\begin{table}[ht]
\centering
\begin{tabular}{ccc|cc}
\toprule
 & \multicolumn{2}{c}{\textbf{Our approach}} & \multicolumn{2}{|p{5cm}}{\centering \textbf{Cutting-plane approach}\newline\textbf{ of \cite{zinchenko2010numerical,girard2015erratum}}} \\
$n = n_A\times n_B$ & value & time (s) & value & time (s)\\
\midrule 
\multirow{5}{*}{$4 = 2\times 2$}
& 0.0670 & 0.68 s & [0.0669,0.0678] & 6.38 s\\
& 0.0002 & 0.54 s & [0.0000,0.0010] & 4.33 s\\
& 0.0157 & 0.52 s & [0.0150,0.0160] & 6.58 s\\
& 0.0478 & 0.52 s & [0.0473,0.0480] & 7.45 s\\
& 0.0027 & 0.60 s & [0.0020,0.0030] & 6.70 s\\ \hline
\multirow{5}{*}{$6 = 3\times 2$}
& 0.0088 & 0.63 s & [0.0083,0.0093] & 14.21 s\\
& 0.0052 & 0.69 s & [0.0047,0.0057] & 17.28 s\\
& 0.0476 & 0.63 s & [0.0473,0.0483] & 17.40 s\\
& 0.0133 & 0.64 s & [0.0130,0.0140] & 14.69 s\\
& 0.0169 & 0.62 s & [0.0166,0.0174] & 26.28 s\\ \hline
\multirow{5}{*}{$9 = 3\times 3$}
& 0.0109 & 1.04 s & [0.0105,0.0115] & 44.65 s\\
& 0.0342 & 1.02 s & [0.0339,0.0349] & 39.47 s\\
& 0.0062 & 1.01 s & [0.0056,0.0066] & 52.37 s\\
& 0.0278 & 1.05 s & [0.0276,0.0286] & 40.20 s\\
& 0.0249 & 1.01 s & [0.0247,0.0257] & 26.35 s\\
\bottomrule
\end{tabular}
\caption{Solving \eqref{eq:REE} for random choices of bipartite states $\rho$
of size $n=n_A\times n_B$. In our approach we replace the cost function
$S(\rho\|\tau)$ by the Tsallis relative entropy $S_t(\rho\|\tau)$ with
$t=2^{-8}$ and use the SDP formulations given in this paper. Note that $S_t(\rho\|\tau) \leq S(\rho\|\tau)$ for any $\rho,\tau$ and as such our approach always returns a lower bound on \eqref{eq:REE}. We see that on all
the matrices tested, our method is much faster than the cutting-plane approach
of \cite{zinchenko2010numerical,girard2015erratum} and gives better lower
bounds. Note that the approach of
\cite{zinchenko2010numerical,girard2015erratum} returns an interval of length
$\epsilon$ guaranteed to contain the optimal value of \eqref{eq:REE} (we set
$\epsilon = 10^{-3}$ in the experiments).}
\label{tbl:ree_results}
\end{table}

\section{Conclusion}
\label{sec:conclusion}
% !TEX root = sdp_rational_lieb_LAA.tex

We conclude by discussing the possibility of a SDP representation for a related jointly convex/concave function.

\paragraph{Sandwiched R{\'e}nyi divergence} The \emph{sandwiched R{\'e}nyi divergence} introduced in \cite{muller2013quantum,wilde2014strong} is defined as
\begin{equation}
\label{eq:sandrelentr}
(A,B) \mapsto \tr \left[ \left(A^{\frac{1-t}{2t}} B A^{\frac{1-t}{2t}}\right)^t \right].
\end{equation}
In \cite{frank2013monotonicity} Frank and Lieb proved that \eqref{eq:sandrelentr} is jointly concave for $t \in [1/2,1]$ and jointly convex for $t \geq 1$. Note that if $A$ and $B$ commute then \eqref{eq:sandrelentr} reduces to $\tr\left[A^{1-t} B^t\right]$; however these two expressions are different for general noncommuting matrices $A$ and $B$. The quantity \eqref{eq:sandrelentr} has found applications in quantum information theory, see e.g., \cite{tomamichelbook}.
In the case $t=1/2$, the expression \eqref{eq:sandrelentr} is called the \emph{fidelity} of $A$ and $B$ and is known to have the following semidefinite programming formulation \cite[Section 3.2]{watrousTQIbook}:
\[
\tr\left[\left(A^{1/2} B A^{1/2}\right)^{1/2}\right] = \max_{Z \in \CC^{n\times n}} \frac{1}{2}\left(\tr[Z]+\tr[Z^*]\right) \; : \; \begin{bmatrix} A & Z\\ Z^* & B\end{bmatrix} \succeq 0.
\]
A natural question is:
\begin{problem}
Find a semidefinite programming formulation for \eqref{eq:sandrelentr} for any $t \geq 1/2$ rational.
\end{problem}

\section*{Acknowledgments}

Hamza Fawzi was supported in part by AFOSR FA9550-11-1-0305. James Saunderson was supported by NSF grant CCF-1409836.

Hamza Fawzi would like to thank Omar Fawzi for discussions and comments, and for pointing out a mistake in Lemma 4 in a previous version of this manuscript.

\appendix
\section{Construction for the matrix geometric mean}
\label{app:construction}
% !TEX root = sdp_rational_lieb_arXiv_v3.tex

In this section we give an SDP description of the matrix geometric mean.
Our construction heavily relies on the properties of the geometric mean which we review below.

\subsection{Properties of the matrix geometric mean}

%relatively prime integers $p$ and $q$ with $q\neq 0$
\if0
\paragraph{The Schur complement lemma} 
The following is a well-known property of the Schur complement stated for the specific case we require and will be useful later.
\begin{lemma}
	If $A,B,C \in \S^n$ with $C \in \S^n_{++}$ then
	\[ \begin{bmatrix} A & B\\B & C\end{bmatrix} \psd 0 \qquad\iff \qquad A - BC^{-1}B \psd 0. \]
\end{lemma} 
\fi
%If $K \in \CC^{n\times m}$ we denote by $\vecm(K)$ the column vector of size $nm$ obtained by concatenating the rows of $K$. Then we have the following identity which holds for any $A \in \S^n,B \in \S^m$:
%\begin{equation}
%\label{eq:kronvecid}
%\tr\left[K^* A K B\right] = \vecm(K)^* (A\otimes \overline{B}) \vecm(K)
%\end{equation}
%where $\overline{B}$ is the entrywise complex conjugate of $B$.

For convenience, we first recall the definition of the \emph{$t$-weighted geometric mean} \mbox{$G_{t}: \S_{++}^n\times \S_{++}^n\rightarrow \S_{++}^n$}:
\[ G_{t}(A,B) = A \#_t B := A^{1/2}\left(A^{-1/2}BA^{-1/2}\right)^tA^{1/2}.\]

The following lemma summarizes important and well-known properties of the weighted geometric mean used in our construction.%, see e.g., \cite[Lemma 2.1]{lawson2013weighted}.
\begin{lemma}
\label{lem:properties}
Suppose $A,B \in \S^n_{++}$.
\begin{itemize}
\item[(i)] If $X$ is an $n\times n$ invertible matrix and $t\in [0,1]$ then $X(A\#_t B)X^* = (XAX^*) \#_t (XBX^*).$
\item[(ii)] (Monotonicity) If $A \psd B \psd 0$ and $C \psd D \psd 0$  and $t\in [0,1]$ then $A \#_t C \psd B \#_t D$.
\item[(iii)] For any $s,t \in \RR$
	\begin{align}
	A \#_t B & = B \#_{1-t} A \label{eq:rev}\\
	A \#_s (A \#_t B) & = A \#_{st} B \label{eq:st}\quad\text{and}\\
	(A \#_t B) \#_s B & = A \#_{s+t-st} B.\label{eq:revst}
\end{align}
\item[(iv)] For any $s,t\in \RR$, and any $X \in \S_{++}^n$, 
\begin{equation}
	\label{eq:ineq-equiv} X \#_s A \psd X \#_t B \;\;\iff\;\; X \#_{-s} A \nsd X \#_{-t} B \;\;\iff\;\; A \#_{s+1} X \nsd B \#_{t+1} X.
\end{equation}
\end{itemize}
\end{lemma}
\begin{proof}
Properties (i)-(iii) are well-known, see e.g., \cite[Lemma 2.1]{lawson2013weighted}. We only include a proof of (iv).
By first multiplying on the left and right by $X^{-1/2}$, then inverting both sides, then multiplying on the left and right by $X^{1/2}$
we have that 
\begin{align*}
	X \#_s A \psd X \#_t B & \iff (X^{-1/2}AX^{-1/2})^s \psd (X^{-1/2}BX^{-1/2})^t\\
				& \iff (X^{-1/2}BX^{-1/2})^{-t} \psd (X^{-1/2}AX^{-1/2})^{-s}\\
				& \iff X \#_{-t} B \psd X \#_{-s} A.
\end{align*}
Finally it follows from~\eqref{eq:rev} that $X \#_{-t} B \psd X \#_{-s} A$ is equivalent to $ B\#_{t+1} X \psd A \#_{s+1} X$. 
%\end{itemize}
\end{proof}

The properties given in Lemma \ref{lem:properties} can be directly translated to relationships between the hypographs/epigraphs of the matrix geometric mean. Recall that $\hypo_t$ and $\epi_t$ are defined as:
\[
\hypo_{t} := \left\{ (A,B,T) \in \S^n_{++} \times \S^n_{++} \times \S^n : A \#_t B \succeq T \right\}
\]
for $t \in [0,1]$, and
\[
\epi_t := \left\{ (A,B,T) \in \S^n_{++} \times \S^n_{++} \times \S^n : A \#_t B \preceq T \right\}
\]
for $t \in [-1,0] \cup [1,2]$.

\begin{lemma}
\label{lem:epihyporel}
The following holds:
\begin{itemize}
\item[(i)] If $t \in [0,1]$ then
\begin{equation}
\label{eq:hyporev}
\hypo_{1-t} = \left\{(A,B,T) : (B,A,T)\in \hypo_{t}\right\}.
\end{equation}
\item[(ii)] If $t\in [-1,0]\cup[1,2]$ then 
\begin{equation}
\label{eq:epirev}
\epi_{1-t} = \left\{(A,B,T) : (B,A,T)\in \epi_{t}\right\}.
\end{equation}
\item[(iii)] For any $s,t \in [0,1]$ we have
\begin{equation}
\label{eq:simple-1}
\hypo_{st} = \{(A,B,T): \exists Z \;\;\text{s.t.}\;\;
			(A,B,Z)\in \hypo_{t},\;\;(A,Z,T)\in \hypo_{s}\}.
\end{equation}
\item[(iv)] For any $t\in [0,1]$, 
\begin{equation}
\label{eq:negt}
 \epi_{-t} = \left\{(A,B,T): \exists S \;\;\text{s.t.}\;\;(A,B,S)\in \hypo_{t},\;\; \begin{bmatrix} T & A\\A & S \end{bmatrix} \psd 0\right\}.
\end{equation}
\end{itemize}
\end{lemma}
\begin{proof}
The proof of this lemma is a direct consequence of the properties of the matrix geometric mean stated in Lemma \ref{lem:properties}. We include a proof of (iv), the other items can be proved in a similar way. First observe that for any $A,S$ positive definite we have $A\#_{-1} S = AS^{-1}A$ thus by the Schur complement lemma we have
\begin{equation}
\label{eq:epiminusone}
\begin{bmatrix}T & A\\ A & S\end{bmatrix} \succeq 0 \quad \iff \quad A\#_{-1} S \preceq T.
\end{equation}
%(i), (ii) Equations \eqref{eq:hyporev} and \eqref{eq:epirev} follow directly from the fact that $A \#_{1-t} B = A \#_{t} B$ (Theorem~\ref{thm:properties}).
%(iii) To prove \eqref{eq:simple-1} suppose $A \#_{st} B \psd T$. Then if we let $Z = A \#_t B$ we have $(A,B,Z)\in \hypo_t$ and 
%\[ A \#_s Z = A\#_s (A\#_t B) = A \#_{st} B \psd T\]
%so $(A,Z,T)\in \hypo_{s}$. Conversely, suppose there exists $Z\in \S_{++}^n$ such that $A \#_t B \psd Z$ and $A \#_s Z \psd T$. 
%	Then by monotonicity of the weighted geometric mean (Theorem~\ref{thm:properties}) it follows that 
%	$A \#_{st} B = A \#_s (A \#_t B) \psd A \#_s Z \psd T$ and so that $(A,B,T)\in \hypo_{st}$. 
%(iv) To prove \eqref{eq:negt}, 
To prove \eqref{eq:negt}, suppose $(A,B,T) \in \epi_{-t}$, i.e., $A\#_{-t} B \nsd T$. Let $S = A\#_t B$. Then $A \#_{-1}S = A \#_{-1}(A \#_t B) = A \#_{-t} B \nsd T$. So, by \eqref{eq:epiminusone} we have
	\[ \begin{bmatrix} T & A\\A & S\end{bmatrix} \psd 0\] 
as desired.
	For the reverse inclusion, suppose there exists $S\in \S_{++}^n$ such that $A\#_t B \psd S$ and $S\psd A\#_{-1}T$. 
	Then by~\eqref{eq:ineq-equiv} of Lemma~\ref{lem:properties} we have that 
	\[ A \#_t B \psd A \#_{-1} T \;\;\implies\;\; A\#_{-t} B \nsd A \#_1 T = T\]
 	Hence $(A,B,T)\in \epi_{-t}$ as required.
\end{proof}

\subsection{Semidefinite representation of the matrix geometric mean}

The remainder of this section is devoted to getting an SDP description of the matrix geometric mean (Theorem \ref{thm:mgeomean-sdp}). Our construction is recursive in nature and heavily relies on the properties of the matrix geometric mean given above. Section \ref{sec:basecase} treats the base case $t=1/2$, Section \ref{sec:denp2} treats the case where $t=p/q \in [0,1]$ and $q$ is a power of two and Section \ref{sec:nump2} treats the case where $t = p/q \in [1/2,1]$ and $p$ is a power of two. In Section \ref{sec:summary} we combine these cases together and complete the proof of Theorem \ref{thm:mgeomean-sdp}.

\subsubsection{Base case $t=1/2$}
\label{sec:basecase}

The following well-known lemma gives an SDP description of $\hypo_{1/2}$. It forms the base case of our construction and we thus include a proof for completeness.

\begin{lemma}
\label{lem:base-case}
We have
\begin{equation}
\label{eq:hypo1/2}
\hypo_{1/2} = \left\{(A,B,T) : \exists W \in \S^n, \;\; \begin{bmatrix} A & W\\W & B\end{bmatrix} \psd 0 \text{ and } W \psd T\right\}.
\end{equation}
%and
%\begin{equation}
%\label{eq:epiminusone}
%\epi_{-1} = \left\{(A,B,T) : \begin{bmatrix} T & A\\A & B\end{bmatrix}\psd 0\right\}.
%\end{equation}
\end{lemma}
\begin{proof}
For the inclusion $\supseteq$ we use the Schur complement lemma and monotonicity of the matrix square root:
\[
\begin{aligned}
\begin{bmatrix} A & W\\ W & B\end{bmatrix} \succeq 0 \Longrightarrow B \succeq WA^{-1}W & \Longrightarrow A^{-1/2} B A^{-1/2} \succeq (A^{-1/2} T A^{-1/2})^2\\
&\Longrightarrow (A^{-1/2} B A^{-1/2})^{1/2} \succeq A^{-1/2} W A^{-1/2}\\
&\Longrightarrow A \#_{1/2} B \succeq W
\end{aligned}
\]
which implies, with the condition $W \psd T$ that $A \# B \psd T$.

For the reverse inclusion $\subseteq$, it suffices to take $W = A\# B$. Indeed one can verify that $\left[\begin{smallmatrix} A & A\#B\\ A\#B & B \end{smallmatrix}\right]$ is positive semidefinite since we have: 
\[
\begin{bmatrix}
A & A\# B\\
A \# B & B
\end{bmatrix}
=
\begin{bmatrix}
A^{1/2}\\
A^{1/2} (A^{-1/2} B A^{-1/2})^{1/2}
\end{bmatrix}
\begin{bmatrix}
A^{1/2}\\
A^{1/2} (A^{-1/2} B A^{-1/2})^{1/2}
\end{bmatrix}^*
\psd 0. \]
%Equation \eqref{eq:epiminusone} easily follows the Schur complement lemma and  the observation that $A \#_{-1} B = A B^{-1} A$.
\end{proof}

\begin{remark}
In the description \eqref{eq:hypo1/2} $(A,B,T)$ is understood to be restricted to $\S^n_{++} \times \S^n_{++} \times \S^n$. This will be implicit in our subsequent SDP descriptions of $\hypo_t$ and $\epi_t$.
\end{remark}

\subsubsection{Denominator is a power of two}
\label{sec:denp2}
In this section we give an SDP description of $\hypo_{p/2^{\ell}}$ when $p$ is odd and $\ell$ is a positive integer such that $p<2^{\ell}$.
Observe that $p/2^{\ell}$ has a binary expansion of length $\ell$ as $(0.m_\ell m_{\ell-1}\cdots {m_{1}})_2$ where $m_1=1$ (because $p$ is odd) and 
$m_i\in \{0,1\}$ for $i=2,\ldots,\ell$. The construction can be expressed explicitly in terms of this binary expansion by repeatedly applying the following recursive rules, which follow easily from the properties of the matrix geometric mean (Lemma \ref{lem:properties}): for $t \in [0,1/2]$, we have
\[
\hypo_t = \left\{(A,B,T): \exists Z\in \S^n\;\text{s.t.}\; (A,Z,T)\in \hypo_{2t},\; \begin{bmatrix} A & Z\\Z & B \end{bmatrix} \psd 0\right\},
\]
and for $t \in [1/2,1]$,
	\[
	\hypo_t = \left\{(A,B,T): \exists Z\in \S^n\;\text{s.t.}\; (Z,B,T)\in \hypo_{2t-1},\; \begin{bmatrix} A & Z\\Z & B \end{bmatrix} \psd 0\right\}.
	\]

Proposition~\ref{prop:denp2-hypo}, to follow, explicitly gives this semidefinite formulation of $\hypo_{p/2^{\ell}}$. 
Note that if $m = 0$ then $A \#_m B = A$ and if $m=1$ then $A\#_m B = B$. In particular, in each case the expression is actually linear in $A$ and $B$. 
\begin{proposition}
	\label{prop:denp2-hypo}
	Suppose $p$ is an odd positive integer and $\ell$ is a positive integer such that $p< 2^\ell$. 
	Let $p/2^{\ell} = (0.m_{\ell}m_{\ell-1}\cdots m_{1})_2$ be the binary expansion of $p/2^{\ell}$ where $m_1=1$ and  $m_i\in \{0,1\}$
	for $i=2,\ldots,\ell$. Then 
	\begin{align}
		 \hypo_{p/2^\ell} & = \left\{(A,B,T): \exists Z_1,\ldots,Z_{\ell-1},Z_{\ell}\in \S^n\;\text{s.t.}\;\;
	\begin{bmatrix} A \#_{m_i} B & Z_{i}\\Z_{i} & Z_{i-1}\end{bmatrix} \psd 0\;\;\text{for $i=2,3,\ldots,\ell$,}\right.\nonumber\\
		& \qquad\qquad\qquad\qquad\qquad\qquad\qquad\qquad\qquad\qquad
	\left.\begin{bmatrix} A & Z_{1}\\Z_{1} & B\end{bmatrix} \psd 0, \; Z_{\ell} \psd T\right\}.\label{eq:LMIden2}
	\end{align}

	Hence $\hypo_{p/2^\ell}$ has an SDP description with $\ell$ LMIs, each of size $2n\times 2n$, and one LMI of size $n\times n$. 
\end{proposition}
 \begin{proof}
 For the inclusion $\subseteq$, take $Z_1 = A\# B$ and $Z_{i} = (A \#_{m_i} B) \# Z_{i-1}$ for $i=2,\ldots,\ell$. Using properties \eqref{eq:st} and \eqref{eq:revst} of the matrix geometric mean one can verify (e.g., by induction) that $Z_i = A\#_{0.m_{i} \dots m_1} B$ for all $i=1,\ldots,\ell$ and in particular $Z_{\ell} = A \#_{p/2^{\ell}} B \psd T$.
 
 For the reverse inclusion $\supseteq$, first note that an LMI of the form $\left[\begin{smallmatrix} X & Z\\Z & Y\end{smallmatrix}\right] \psd 0$ implies $X \# Y \psd Z$ (see first part of the proof of Lemma \ref{lem:base-case}). Thus the LMI constraints on the right-hand side of \eqref{eq:LMIden2} imply
 \begin{equation}
 \label{eq:LMIden2-2gm}
 \begin{cases}
  A \# B \psd Z_1, \text{ and } &\\
 (A \#_{m_i} B) \# Z_{i-1} \psd Z_i & \text{ for } i=2,\ldots,\ell.
 \end{cases}
 \end{equation}
From \eqref{eq:LMIden2-2gm} it follows by induction on $i$, and properties of the matrix geometric mean (Equations \eqref{eq:st} and \eqref{eq:revst}) that $A\#_{0.m_i m_{i-1} \dots m_1} B \psd Z_i$ for all $i=1,\ldots,\ell$. In particular this implies that $A\#_{p/2^{\ell}} B \psd Z_{\ell} \psd T$.
	\end{proof}	

We conclude with an example in which the denominator is a power of two.
\begin{example}[SDP representation of $\hypo_{5/8}$]
	\label{eg:5over8}
	Let $p=5$ and $\ell=3$ so that $p/2^{\ell} = 5/8 = (0.101)_2$. Consider constructing a SDP representation of $\hypo_{5/8}$. 
	We have that $m_1 = m_3=1$ and $m_2=0$ so that $A \#_{m_1} B = B$ and $A \#_{m_2} B = A$. Applying Proposition~\ref{prop:denp2-hypo} gives
	\[ \hypo_{5/8} = \left\{(A,B,T): \exists Z_1,Z_2,Z_3\;\;\text{s.t.}\;\;
	Z_3 \psd T, \;\;
\begin{bmatrix} B & Z_3\\Z_3 & Z_2\end{bmatrix} \psd 0,\;\;\begin{bmatrix} A & Z_2\\Z_2 & Z_1\end{bmatrix} \psd 0,\;\;
\begin{bmatrix} A & Z_1\\Z_1 & B\end{bmatrix} \psd 0\right\}\]
using $\ell=3$ LMIs of size $2n\times 2n$ and one LMI of size $n\times n$. 
\end{example}

\subsubsection{Numerator is a power of two}
\label{sec:nump2}

In this section we show how to construct an SDP representation of $\hypo_{t}$ when $t$ has a numerator that is a power of two and $t \in [1/2,1]$.
%positive integer and $\ell$ is the largest integer satisfying $2^\ell < q$ (i.e.\ $\ell = \lfloor \log_2(q)\rfloor$). 
We do this by relating $\hypo_{t}$ and $\hypo_{\frac{2t-1}{t}}$ (see Lemma~\ref{lem:simple-2} to follow). This is useful because if $t = 2^{\ell}/q$ with $t \in [1/2,1]$, then $\frac{2t-1}{t} = \frac{2^{\ell+1}-q}{2^{\ell}}$
has a denominator that is a power of two. Hence we can relate $\hypo_{2^\ell/q}$ with $\hypo_{\frac{2^{\ell+1}-q}{2^{\ell}}}$, an SDP description of 
which we can obtain from Proposition~\ref{prop:denp2-hypo}. 

%In this section we consider the problem of constructing a SDP representation of $\hypo_{2^\ell/q}$ when $q$ is an odd positive integer and 
%$\ell$ is the largest integer satisfying $2^{\ell} < q$ (i.e.\ $\ell = \lfloor \log_2(q)\rfloor$). 
%Under these assumptions, we show that $\hypo_{2^{\ell}/q}$
%can be expressed in terms of $\hypo_{\frac{2^{\ell+1}-q}{2^{\ell}}}$. By applying the results of Section~\ref{sec:denp2} (in particular Proposition~\ref{prop:denp2-hypo})
%we can obtain a SDP representation of $\hypo_{\frac{2^{\ell+1}-q}{2^{\ell}}}$ involving $\ell$ LMIs, each of size $2n\times 2n$. 
%The following result allows us to express for any $t \in [1/2,1]$, the set
%$\hypo_{t}$ in terms of $\hypo_{\frac{2t-1}{t}}$. This is useful because if $t$
%is a rational and has a numerator that is a power of two, then $\frac{2t-1}{t}$
%is a rational with a denominator that is the same power of two, and so in this
%case one can use the construction from the previous section to obtain a SDP
%representation of $\hypo_{t}$.

\begin{lemma}
\label{lem:simple-2}
If $t \in [1/2,1]$ then
\begin{equation}
\label{eq:simple2s}
 \hypo_{t} = \left\{(A,B,T): \exists Z,W\in \S^n_{++}\;\;\text{s.t.}\;\;
			(A,W,Z)\in \hypo_{\frac{2t-1}{t}},\;\;\begin{bmatrix} Z & W\\W & B\end{bmatrix} \psd 0,\;\;W\psd T\right\}.
\end{equation}
\end{lemma}
\begin{proof}
We first prove $\subseteq$. Suppose $A \#_{t} B \psd T$. Then let $Z = A \#_{2t-1} B$ and $W = A \#_{t} B$. It is easy to see that the conditions on the right-hand side of \eqref{eq:simple2s} are satisfied. Indeed first we have
\[
A \#_{\frac{2t-1}{t}} W = A \#_{\frac{2t-1}{t}} (A \#_t B) = A \#_{2t-1} B = Z
\]
and this shows that $(A,W,Z) \in \hypo_{\frac{2t-1}{t}}$. Second, using Property~\eqref{eq:revst} and Lemma~\ref{lem:base-case} we have, 
\[
Z \#_{1/2} B = (A\#_{2t-1} B)\#_{1/2} B = A\#_t B = W
\quad\text{which implies that}\quad
\begin{bmatrix} Z&W\\W&B\end{bmatrix} \psd 0.
\]
%\[W \#_{-1} B = (A \#_{t} B) \#_{-1} B = A\#_{2t-1} B = Z \quad\text{which implies that}\quad \begin{bmatrix} Z&W\\W&B\end{bmatrix} \psd 0.\]
Finally we have that $W = A \#_t B \psd T$ by assumption.

	We now prove $\supseteq$. Suppose there exist $Z,W\in \S_{++}^n$ such that $A \#_{\frac{2t-1}{t}} W \psd Z$ and $W \#_{-1} B \nsd Z$ and $W \psd T$. 
	Then since $1-\frac{2t-1}{t} = \frac{1}{t}-1$ we have that $W \#_{1/t-1} A \psd Z$. Then
	\[ W \#_{1/t-1} A \psd Z \psd W \#_{-1} B.\]
	Applying~\eqref{eq:ineq-equiv} from Lemma~\ref{lem:properties} it follows that 
	\[B =  B \#_{-1+1} W \psd A \#_{1/t-1+1} W = A \#_{1/t} W.\]
	Then since $t\in [1/2,1]$ and $G_t$ is monotone for $t\in [0,1]$, applying $G_t(A,\cdot)$ to both sides gives 
	\[ A \#_t B \psd A \#_t (A\#_{1/t} W) = A \#_1 W = W \psd T\]
	as required.
	\end{proof}

% 	Let $\tilde{W} = A^{-1/2}WA^{-1/2}$, let $\tilde{Z} = A^{-1/2}ZA^{-1/2}$,
% 	and let $\tilde{B} = A^{-1/2}BA^{-1/2}$. Then by definition of the weighted geometric mean, inequality $A\#_{\frac{2s-1}{s}} W \succeq Z$ reads
% 	\[
% 	A^{1/2} (A^{-1/2} W A^{-1/2})^{\frac{2s-1}{s}} A^{1/2} \succeq Z
% 	\]
% 	i.e.,
% 	\begin{equation}
% 		\tilde{W}^{\frac{2s-1}{s}} \psd \tilde{Z} \qquad\text{and so}\qquad \tilde{Z}^{-1} \psd \tilde{W}^{\frac{1-2s}{s}}.\label{eq:WZ}
% 	\end{equation}
% 	Furthermore, $Z \# B \psd W$ implies that
% 	\begin{equation}	
% 		 \begin{bmatrix} \tilde{Z} & \tilde{W}\\\tilde{W} & \tilde{B}\end{bmatrix} \psd 0\qquad\text{and so, by the Schur complement lemma,}\qquad 
% 		\tilde{B} \psd \tilde{W}\tilde{Z}^{-1}\tilde{W}.\label{eq:ZWBsc}
% 	\end{equation}
% 	Combining~\eqref{eq:ZWBsc} and~\eqref{eq:WZ} we obtain $\tilde{B} \psd \tilde{W}^{\frac{1-2s}{s}+2} = \tilde{W}^{\frac{1}{s}}$. Since $0\leq s \leq 1$, the map $X\mapsto X^s$ is monotone and so we get that
% 	$\tilde{B}^{s} \psd \tilde{W}$. Multiplying on the left and right by $A^{1/2}$ on both sides we obtain
% 	$A \#_{s} B \psd W$. Since $W \psd T$ by assumption it follows that $(A,B,T)\in \hypo_{s}$ as required.
% \end{proof}
%\cbend

Note that if $t = 2^{\ell}/q$ then $\frac{2t-1}{t} = \frac{2^{\ell+1} - q}{2^{\ell}}$ is a dyadic number and so $\hypo_{\frac{2t-1}{t}}$ has a SDP description from the previous section (Proposition \ref{prop:denp2-hypo}).
%We now use Lemma~\ref{lem:simple-2} to give a PSD lift of $\hypo_{2^{\ell}/q}$ when $\ell = \lfloor \log_2(q)\rfloor$. 
\begin{proposition}
	\label{prop:nump2-hypo}
	Assume $\ell,q$ are integers such that $\frac{2^{\ell}}{q} \in [1/2,1]$. Then
%	If $q$ is an odd positive integer and $\ell$ is the largest integer such that $2^{\ell} < q$ (i.e.\ $\ell = \lfloor \log_2(q)\rfloor$) then 
% \marginparsmall{Use $(W,A,Z)$ or $(A,W,Z)$? In \eqref{eq:simple2s} it is $(A,W,Z)$}
\begin{equation}
\label{eq:SDPnump2}
 \hypo_{2^{\ell}/q} = \left\{(A,B,T): \exists Z,W\;\;\text{s.t.}\;\;(A,W,Z)\in \hypo_{\frac{2^{\ell+1}-q}{2^{\ell}}},\;\; 
	\begin{bmatrix} Z & W\\W & B \end{bmatrix} \psd 0,\;\;W \psd T\right\}.
\end{equation}
Hence $\hypo_{2^{\ell}/q}$ has a SDP representation using $\ell+1$ LMIs of size $2n\times 2n$ and one LMI of size $n\times n$.
\end{proposition}
\begin{proof}
The SDP description follows directly from Lemma \ref{lem:simple-2} with $t = \frac{2^{\ell}}{q}$.
%	First observe that since $2^{\ell}/q \in [1/2,1]$ we have $q < 2^{\ell+1}$ and so that $0 < q-2^{\ell} < 2^{\ell}$. Hence if we let $t = \frac{q-2^{\ell}}{2^{\ell}}$ we have that $0 < t < 1$ and $1/(1+t) = \frac{2^{\ell}}{q}$.	As such, the description of $\hypo_{2^{\ell}/q}$ in terms of $\hypo_{\frac{q-2^{\ell}}{2^{\ell}}}$ follows directly from Lemma~\ref{lem:simple-2}. 
Since $\hypo_{\frac{2^{\ell+1}-q}{2^{\ell}}}$ has a SDP description with $\ell$ LMIs of size $2n \times 2n$ (cf. Proposition \ref{prop:denp2-hypo}) the conclusion about the size of the description \eqref{eq:SDPnump2} holds.
%	The conclusion about the SDP description for $\hypo_{2^{\ell}/q}$ holds because if $q$ is odd then $q-2^{\ell}$ is also odd and, in addition, $q-2^{\ell} < 2^{\ell}$. 	Hence Proposition~\ref{prop:denp2-hypo} gives a PSD lift of $\hypo_{\frac{q-2^{\ell}}{2^{\ell}}}$ using $\ell$ LMIs each of size $2n\times 2n$. 	Overall we obtain a PSD lift of $\hypo_{2^{\ell}/q}$ of using $\ell+1$ LMIs each of size $2n\times 2n$ and one LMI of size $n\times n$. 
\end{proof}	
We conclude with an example in which the numerator is a power of two.
\begin{example}[SDP representation of $\hypo_{8/13}$]
	\label{eg:8over13}
	Let $q=13$ and $\ell=3$ so that $2^{\ell}/q = 8/13$. Note that $8/13\in [1/2,1]$.
	Consider constructing an SDP description of $\hypo_{8/13}$. We have that $(2^{\ell+1}-q)/2^{\ell} = 3/8 = (0.011)_2$. Hence,
	by Proposition~\ref{prop:nump2-hypo},
	\[ \hypo_{8/13} = \left\{(A,B,T): \exists Z_3,W\;\;\text{s.t.}\;\;(A,W,Z_3)\in \hypo_{3/8},\;\;\begin{bmatrix} Z_3 & W\\W & B \end{bmatrix} \psd 0,\;\;
	W \psd T\right\}.\]
	Using Proposition~\ref{prop:denp2-hypo} to obtain a semidefinite description of $\hypo_{3/8}$ gives  
	\begin{align*}
	 \hypo_{8/13} & =\! \left\{(A,B,T): \exists Z_3,W,Z_1,Z_2 \;\text{s.t.}\;
	\begin{bmatrix} A & Z_3\\Z_3 & Z_2\end{bmatrix} \psd 0,\;\begin{bmatrix} W & Z_2\\Z_2 & Z_1\end{bmatrix} \psd 0\;,
	\begin{bmatrix} W & Z_1\\Z_1 & A\end{bmatrix} \psd 0\right.
	\\
	& \qquad\qquad\qquad\qquad\qquad\qquad\qquad\qquad\qquad\qquad\qquad\qquad\qquad\qquad\left.
	\begin{bmatrix} Z_3 & W\\W & B \end{bmatrix} \psd 0,\;\; W \psd T\right\},
	\end{align*}
a SDP representation of $\hypo_{8/13}$ using $\ell+1=4$ LMIs of size $2n\times 2n$ and one LMI of size $n\times n$. 
\end{example}

\subsubsection{Putting everything together and summary of construction}
\label{sec:summary}

We now complete the proof of Theorem \ref{thm:mgeomean-sdp}.

\begin{proof}[Proof of Theorem \ref{thm:mgeomean-sdp}]

First observe that, using relations established in Lemma \ref{lem:epihyporel}, we only need to consider the case $p/q \in [0,1/2]$: indeed if we have an SDP representation of $\hypo_{t}$ for $t \in [0,1/2]$ then we can use the relationship between $\hypo_{1-t}$ and $\hypo_t$ in \eqref{eq:hyporev} to get an SDP representation for $\hypo_t$ in the range $t \in [1/2,1]$ with no additional LMIs. Then using the relationship \eqref{eq:negt} between $\epi_{-t}$ and $\hypo_t$ we can get an SDP representation of $\epi_t$ for $t \in [-1,0]$ with the addition of a single $2n \times 2n$ LMI. Finally using again the relationship \eqref{eq:epirev} between $\epi_{1-t}$ and $\epi_{t}$ we get an SDP representation for $\epi_t$ where $t \in [1,2]$.

It thus remains to prove the case where $t$ is an arbitrary rational in $[0,1/2]$. We show how to do this using the results from the two previous sections. If $t = p/q \in [0,1/2]$ we decompose $t$ as $t = (p / 2^{\ell}) \cdot (2^{\ell} / q)$ where $\ell = \lfloor \log_2(q) \rfloor$. By applying Propositions \ref{prop:denp2-hypo} and \ref{prop:nump2-hypo} to construct respectively $\hypo_{p/2^{\ell}}$ and $\hypo_{2^{\ell}/q}$ and appealing to \eqref{eq:simple-1} we get an SDP description of $\hypo_{t}$ (note that $2^{\ell}/q \in [1/2,1]$ since $\ell = \lfloor \log_2(q) \rfloor$ and so Proposition \ref{prop:nump2-hypo} applies to get an SDP description of $\hypo_{2^{\ell}/q}$).
%To see that our SDP representation has the right size, the SDP representation of $\hypo_{p/2^{\ell}}$ uses at most $\ell$ LMIs of size $2n\times 2n$ and the SDP representation of $\hypo_{2^{\ell}/q}$ uses at most $\ell+1$ LMIs of size $2n\times 2n$ and one LMI of size $n\times n$. Hence our description has at most $2\ell+1 = 2\lfloor \log_2(q) \rfloor + 1$ LMIs of size $2n\times 2n$ and one LMI of size $n\times n$.
%The remainder of this section is thus devoted to the case $t = p/q \in [0,1/2]$. The SDP representation in this case is established as follows:
%\begin{itemize}
%\item \textit{Denominator power of two:} We first show in Proposition \ref{prop:denp2-hypo} how to obtain an SDP representation of $\hypo_t$ in the case where $t$ is a rational number with a denominator that is a power of two.
%\item \textit{Numerator power of two:} We then show in Proposition \ref{prop:nump2-hypo} how to treat the case where $t$ has a numerator that is a power of two and $t \in [1/2,1]$.
%\end{itemize}
%To treat the general case where $t = p/q \in [0,1/2]$ we decompose $t$ as $t = (p / 2^{\ell}) \cdot (2^{\ell} / q)$ where $\ell = \lfloor \log_2(q) \rfloor$ and we then appeal to \eqref{eq:simple-1} (note that since $\ell = \lfloor \log_2(q) \rfloor$ we have $2^{\ell} / q \in [0,1/2]$ so we can apply Proposition \ref{prop:nump2-hypo} for $\hypo_{2^{\ell}/q}$).

To see that our SDP representation has the right size, the SDP representation of $\hypo_{p/2^{\ell}}$ uses at most $\ell$ LMIs of size $2n\times 2n$ and the SDP representation of $\hypo_{2^{\ell}/q}$ uses at most $\ell+1$ LMIs of size $2n\times 2n$ and one LMI of size $n\times n$. Hence our description has at most $2\ell+1 = 2\lfloor \log_2(q) \rfloor + 1$ LMIs of size $2n\times 2n$ and one LMI of size $n\times n$. The size of the SDP representation for the epigraph case $t \in [-1,0] \cup [1,2]$ requires an additional $2n \times 2n$ LMI which comes from identity \eqref{eq:negt}.
\end{proof}

%A summary of our construction is given at the end of this section, in Table \ref{tbl:summaryconstruction}.

%The steps of our recursions are enabled by the following simple result, a consequence of the monotonicity of the weighted geometric mean and the fact that $ A\#_{st} B = A\#_s (A\#_t B)$ (cf. Equation \eqref{eq:st}).

Table \ref{tbl:summaryconstruction} summarizes our SDP construction of the hypograph/epigraph of the matrix geometric mean for arbitrary rationals $t=p/q \in [-1,2]$.

%To summarize our construction we outline below a pseudocode to implement our construction. The function \verb|matrix_geo_mean_hypo(A,B,T,t)| returns a semidefinite description of $A\#_t B \succeq 0$ for $t \in [0,1]$.

\begin{table}[ht]
\centering
\fbox{
\begin{minipage}{12cm}
\textbf{Semidefinite representation of $\hypo_{t}$ for $t=p/q \in [0,1]$}
\begin{itemize}[leftmargin=0.8cm]
\item[(i)] If $q$ is a power of two\\
{\small Use construction in Proposition \ref{prop:denp2-hypo}.}
\item[(ii)] If $t \in [1/2,1]$ and $p$ is a power of two\\
{\small Use Proposition \ref{prop:nump2-hypo} which expresses $\hypo_{t}$ in terms of the hypograph of a dyadic number, then use (i).}
\item[(iii)] If $t$ is any rational in $[0,1/2]$\\
{\small Express $t$ as $t = (p/2^{\ell}) \cdot (2^{\ell}/q)$ where $q = \lfloor \log_2(q) \rfloor$. Use (i) and (ii) to construct $\hypo_{p/2^{\ell}}$ and $\hypo_{2^{\ell}/q}$ and combine them using \eqref{eq:simple-1} to get $\hypo_t$.}
\item[(iv)] If $t$ is any rational in $[1/2,1]$\\
{\small Use relationship \eqref{eq:hyporev} between $\hypo_{t}$ and $\hypo_{1-t}$ then apply (iii).}
\end{itemize}
\end{minipage}
}

\medskip

\fbox{
\begin{minipage}{12cm}
\textbf{Semidefinite representation of $\epi_{t}$ for $t=p/q \in [-1,0]\cup [1,2]$}
\begin{itemize}[leftmargin=0.8cm]
\item[(i)] If $t \in [-1,0]$\\
{\small Use \eqref{eq:negt} to express $\epi_t$ in terms of $\hypo_{-t}$ and apply box above.}
\item[(ii)] If $t \in [1,2]$\\
{\small Use relationship \eqref{eq:epirev} between $\epi_{t}$ and $\epi_{1-t}$ then apply (i).}
\end{itemize}
\end{minipage}
}
\caption{Semidefinite representation of the matrix geometric mean (Theorem \ref{thm:mgeomean-sdp}).}
\label{tbl:summaryconstruction}
\end{table}

\bibliography{sdp_rational_lieb}
\bibliographystyle{alpha}

\end{document}